\documentclass[reqno,12pt,a4letter]{amsart}
\usepackage{amsmath, amsxtra, amssymb, latexsym, amscd, amsthm,enumerate}
\usepackage{graphicx, color}
\usepackage[active]{srcltx}
\usepackage[utf8]{inputenc}
\usepackage[mathscr]{euscript}
\usepackage{mathrsfs,cite}
\usepackage[english]{babel}
\usepackage{color}
\usepackage{hyperref}
\usepackage[active]{srcltx}
\def\Limsup{\mathop{{\rm Lim}\,{\rm sup}}}

\def\F{Fr\'{e}chet\ }

\def\gph{\mbox{\rm gph}\,}
\def\ox{\bar{x}}
\def\epi{\mbox{\rm epi}\,}

\newtheorem{theorem}{Theorem}[section]

\newtheorem{proposition}[theorem]{Proposition}

\theoremstyle{definition}
\newtheorem{definition}[theorem]{Definition}
\newtheorem{example}[theorem]{Example}

\newtheorem{remark}[theorem]{Remark}

\title{On second-order optimality conditions for $C^{1,1}$ optimization problems via Lagrangian functions}

\author{DUONG THI VIET AN$^{1}$ AND NGUYEN VAN TUYEN$^{2}$}
\address{$^{1}$Department of Mathematics and Informatics, Thai Nguyen University of Sciences, Thai Nguyen 250000, Vietnam}
\email{andtv@tnus.edu.vn}
 
\address{$^2$Department of Mathematics, Hanoi Pedagogical University 2, Xuan Hoa, Phuc Yen, Vinh Phuc, Vietnam}
\email{nguyenvantuyen83@hpu2.edu.vn; tuyensp2@yahoo.com}

\date{\today}



\begin{document}
	
	\maketitle

	\begin{abstract}
This paper focuses on optimality conditions for $C^{1,1}$-smooth optimization problems subject to inequality and equality constraints. By employing the concept of limiting (Mordukhovich) second-order subdifferentials to the Lagrangian function associated with the problem, we derive new second-order optimality conditions for the considered problem. Applications for multiobjective optimization problems are studied as well. These results extend and refine existing results in the literature.
	\end{abstract}
	
\section{Introduction}
The theory of second-order optimality conditions has been a subject of significant interest among researchers, see, e.g.~\cite{An2021,AnYen2021,AnXuYen,Andreani_Martinez_Schuverdt_2007,Ben-Tal1980,Ben-Tal1982,Bonnans_Cominetti_Shapiro_1999,Bonnans_Shapiro_2000,ChieuLeeYen2017,Feng_Li_2020,Ginchev_Ivanov_2008,Gutierrez_Jimenez_Novo_2009,HSN_1984,Huy_Tuyen,Huy-Kim-Tuyen,Hung-Tuan-Tuyen,Khanh-et al-24,Jeyakumar_Luc_1998,Jeyakumar_Wang_1999,Penot1999,Tuyen-Huy-Kim,Yang_1993} and the references therein. Hiriart-Urruty et al. in~\cite{HSN_1984} employed generalized Hessian matrices to derive second-order necessary conditions for a class of differentiable optimization problems. Jeyakumar and Luc~\cite{Jeyakumar_Luc_1998} and Jeyakumar and Wang~\cite{Jeyakumar_Wang_1999} utilized the concept of approximate Hessians to establish necessary and sufficient optimality conditions for unconstrained and constrained optimization problems involving continuously differentiable functions, respectively. Another approach, involving the use of second-order directional derivatives, has been explored by Ginchev and Ivanov~\cite{Ginchev_Ivanov_2008} and Guti\'errez et al.~\cite{Gutierrez_Jimenez_Novo_2009}.

The concept of second-order subdifferential, as introduced by Mordukhovich in~\cite[Remark~2.8]{Mordukhovich_1992}, has emerged as a powerful tool in variational analysis. Defined as the coderivative of the first-order subdifferential mapping, this concept, along with its rich calculus (see~\cite[pp. 121-132]{Mordukhovich_2006}, \cite{Mordukhovich-2024},  \cite{Mo_Ou_2001}), has found applications in diverse areas, including stability and sensitivity analysis, second-order optimality conditions, and characterizations of convexity. For a comprehensive overview of recent research on the theory of second-order subdifferentials and their applications, we refer the reader to the recent monograph by Mordukhovich~\cite{Mordukhovich-2024}. This comprehensive work, consisting of nine interrelated chapters, provides a valuable reference for researchers working in this~area.

By using the second-order subdifferential concepts mentioned above, Chieu et al.~\cite{ChieuLeeYen2017}, Huy and Tuyen \cite{Huy_Tuyen}, Feng and Li~\cite{Feng_Li_2020}, and Nadi and Zafarani~\cite{Nadi_Jafarani2022} have established second-order necessary and sufficient optimality conditions for optimization problems in both finite and infinite-dimensional spaces. 
Namely, in~\cite{ChieuLeeYen2017}, the authors have used the limiting second-order subdifferentials to characterize locally optimal solutions of
$C^{1,1}$-smooth unconstrained minimization problems. In \cite{Huy_Tuyen}, Huy and Tuyen introduced the so-called second-order symmetric subdifferential and derived new second-order necessary and sufficient optimality conditions for a class of differentiable optimization problems with geometric constraints. Thereafter, these results are applied to study second-order optimality conditions for vector optimization problems with inequality constraints \cite{Huy-Kim-Tuyen,Tuyen-Huy-Kim}. By applying the notion of the pseudoconvexity at a point in a direction and the second-order characterizations of (strict and strong) pseudoconvexity, the authors of~\cite{Nadi_Jafarani2022} have obtained second-order optimality conditions (for strict local minima and isolated local minima) in nonlinear programming with continuously differentiable data. Recently, Feng and Li \cite{Feng_Li_2020} established a second-order mean value inequality for $C^{1,1}$ functions, employing the concept of limiting (Mordukhovich) second-order subdifferentials. By applying this inequality, the authors derived novel second-order Fritz John-type necessary and sufficient optimality conditions for inequality-constrained optimization problems.

Our approach departs from~\cite{Feng_Li_2020} by avoiding the use of second-order tangent sets to the feasible set. Instead, we employ the limiting second-order subdifferential of the Lagrangian function to obtain second-order optimality conditions for $C^{1,1}$-smooth constrained optimization problems. Our approach differs significantly from that of~\cite{Feng_Li_2020}.
Firstly, we focus on optimization problems with finite inequality and equality constraints. Secondly, we employ the Lagrangian method and utilize information from the first-order tangent set, whereas the optimality conditions in \cite{Feng_Li_2020} rely on information from the second-order tangent set.
Furthermore, in Theorem~\ref{main_theorem}, we study the second-order necessary optimality conditions for  isolated solutions of order 2 as well.	We also present a counterexample to demonstrate that our theorem is applicable in situations where the result in \cite[Theorem 4.2]{Feng_Li_2020} cannot be applied.

The paper organization is as follows. Section \ref{Section-2} provides some basic definitions and auxiliary results. The second-order necessary optimality conditions are analyzed in Section \ref{Section-3} while second-order sufficient optimality conditions are investigated in Section~\ref{Section-4}. Applications for multiobjective optimization problems are presented in Section~\ref{Section-5}. Some conclusions are given in the final section.

\section{Preliminaries}\label{Section-2}
Throughout the paper, the considered spaces are finite-dimensional Euclidean with the inner product and the norm being denoted by $\langle \cdot, \cdot \rangle$ and by $||\cdot||$, respectively.
%For a set $\Omega$ in $X$, the interior, the closure, and the relative interior of $\Omega$ are denoted by $\Int \Omega$, $\cl \Omega$, and  $\ri \Omega$, respectively. For a linear operator $A$, ${\ker}\, A$ (resp., ${\rge}\,A$) stands for the kernel (resp., the range) of $A$.
Along with single-valued maps usually denoted by $f : \mathbb{R}^n\rightarrow \mathbb{R}^m$, we consider set-valued maps (or multifunctions) $F :  \mathbb{R}^n \rightrightarrows {\mathbb{R}^m}$ with values $F(x)$ in the collection of all the subsets of $\mathbb{R}^m$. The limiting construction
\begin{align*}%\label{PK}
	\Limsup\limits_{x\rightarrow \bar x} F(x):=\bigg\{ y \in \mathbb{R}^m \mid \exists x_k \rightarrow \bar x, y_k \rightarrow y \ \mbox{with}\ y_k\in F(x_k), \forall k=1,2,....\bigg\}
\end{align*}
is known as the \textit{Painlev\'e-Kuratowski outer/upper limit} of $F$ at $\bar x$. All the maps considered below are proper, i.e., $F (x)\not= \emptyset$ for some $x\in \mathbb{R}^n$.

\begin{definition} {\rm (see~\cite[pp. 5--6]{Mordukhovich_2018})}\rm
	Let $\Omega$ be a nonempty subset of $\mathbb{R}^n$ and $\bar x \in \Omega$. The \textit{Fr\'echet (regular) normal cone}  to $\Omega$ at $\bar x$ is defined by
	\begin{align*}
		\widehat N(\bar x, \Omega)=\Big\{ v\in \mathbb{R}^n\mid \limsup\limits_{x \xrightarrow{\Omega}\bar x} \dfrac{\langle v, x-\bar x \rangle}{\|x-\bar x\|} \leq 0 \Big\},
	\end{align*}
	where $x \xrightarrow{\Omega} \bar x$ means that $x \rightarrow \bar x$ and $ x\in \Omega$. The \textit{limiting (Mordukhovich) normal cone} to $\Omega$ at $\bar x$ is given by
	\begin{align*}
		N(\bar x, \Omega)=\Limsup\limits_{ x \xrightarrow{\Omega} \bar x} \widehat{N}(x, \Omega).
	\end{align*}
	We put $\widehat N(\bar x,\Omega)=N(\bar x,\Omega) =\emptyset$ if $\bar x \not\in \Omega$.
\end{definition}	

Clearly, one always has
\begin{align*}
	\widehat N(\bar x,\Omega) \subset N(\bar x,\Omega), \ \forall \Omega \subset \mathbb{R}^n, \forall \bar x \in \Omega.
\end{align*}
If $\Omega$ is convex, one has
\begin{align*}
	\widehat N(\bar x,\Omega) = N(\bar x,\Omega):=\{ v\in \mathbb{R}^n \mid \langle v,x- \bar x \rangle \le 0,  \ \forall x\in \Omega\},
\end{align*}
i.e., the regular normals to $\Omega$  at $\bar x$ coincides with the limiting normal cone and both constructions reduce to the normal cone in the sense of convex analysis.

\begin{definition}\rm  (see, e.g., \cite[Definition 3.41]{Jahn_2011}) Let $\Omega$ be a nonempty subset of $\mathbb{R}^n$. An element $v\in \mathbb{R}^n$ is called a \textit{tangent} to $\Omega$ at a point $\bar x\in \Omega$ if there are a sequence $\{x_k\}\subset \Omega$ and a sequence $\{t_k\}$ of positive real numbers such that $t_k\rightarrow 0^+$, $x_k\rightarrow \bar x$, and $\lim\limits_{k\rightarrow \infty} \dfrac{x_k-\bar x}{t_k}=v.$
\end{definition}	

The set $T(\bar x,\Omega)$ of all tangents to $\Omega$ at $\bar x$ is called the \textit{contingent cone} (or \textit{the Bouligand-Severi tangent cone} \cite[Chapter~1]{Mordukhovich_2006}) to $\Omega$ at $\bar x$. Note that $v\in T(\bar x,\Omega)$ if and only if there exist a sequence $\{t_k\}$ of positive scalars and a sequence of vectors $\{v_k\}$ with $t_k\rightarrow 0^+$ and $v_k \rightarrow v$ as $k\rightarrow \infty$ such that $x_k:=\bar x +t_k v_k$ belongs to $\Omega$ for every $k\in \mathbb{N}.$

In the case, where $\Omega$ is a convex set,  $[T(\bar x, \Omega)]^*=N(\bar x, \Omega)$ and $[N(\bar x, \Omega)]^*=T(\bar x, \Omega)$, where 
$$[N(\bar x, \Omega)]^*:= \{ v\in \mathbb{R}^n \mid \langle v,  x \rangle \le 0, \ \forall x\in N(\bar x, \Omega) \}.$$

\medskip

Let  $F: \mathbb{R}^n\rightrightarrows{\mathbb{R}^m}$ be a set-valued map with the \textit{domain} $${\rm{dom}}\, F:=\{ x \in \mathbb{R}^n \mid F(x)\not=\emptyset\}$$ and the \textit{graph} $${\rm{gph}}\, F:=\{ (x,y) \in \mathbb{R}^n \times \mathbb{R}^m \mid y \in F(x)\}.$$  

\begin{definition}\rm  {\rm (see~\cite[Definition 1.11]{Mordukhovich_2018})}  Let $  F: \mathbb{R}^n\rightrightarrows{\mathbb{R}^m}$ be a multifunction and  $(\bar x, \bar y) \in  {\rm{gph}}\, F$.	The  \textit{limiting (Mordukhovich) coderivative}  of  $F$ at $(\bar x, \bar y)$ is a multifunction $D^* F(\bar x, \bar y): \mathbb{R}^m \rightrightarrows{\mathbb{R}^n}$ with the values
	\begin{align*}
		D^* F(\bar x,\bar y)(v):=&\left\{u\in \mathbb{R}^n \mid (u, -v) \in  N\left((\bar x, \bar y),\gph\, F\right)\right\}, \quad v \in \mathbb{R}^m.
	\end{align*}
	If $(\bar x, \bar y) \notin {\rm{gph}}\, F$, one puts $ D^* F(\bar x, \bar y)(v)=\emptyset$ for any $v\in \mathbb{R}^m$. The symbol $D^* F(\bar x)$ is used when $F$ is single-valued at $\bar x$ and $\bar y=F(\bar x)$. 
\end{definition}

Consider a function $f: \mathbb{R}^n\rightarrow \overline{\mathbb{R}}=\mathbb{R}\cup \{\pm \infty\}$ with the \textit{effective domain} $$\mbox{dom}\, f:=\{x \in \mathbb{R}^n \mid f(x) < +\infty\},$$ the \textit{epigraph} $$ {\rm{epi}}\, f:=\{ (x, \alpha) \in \mathbb{R}^n \times \mathbb{R} \mid \alpha \ge f (x)\},$$
and the \textit{hypergraph} $$ {\rm{hypo}}\, f:=\{ (x, \alpha) \in \mathbb{R}^n \times \mathbb{R} \mid \alpha \le f (x)\}.$$  

\begin{definition}\rm  {\rm (see~\cite[Definition 1.77]{Mordukhovich_2006})}
	Suppose that $\bar {x} \in \mathbb{R}^n$ and $|f(\bar {x})| < \infty.$ 	One calls the set
	\begin{align*}
		\partial f(\bar x):=\left\{v \in \mathbb{R}^n\mid (v, -1) \in  N((\bar {x}, f(\bar{ x})), \epi f)\right\}
	\end{align*}
	the  \textit{limiting (Mordukhovich) subdifferential} of $f$ at $\bar {x}$. If $ |f(\bar x)| = \infty$, one lets  $\partial f(\bar x)$ to be an
	empty set.
\end{definition}

One can use the notion of coderivative to construct the second-order generalized differential theory of extended-real-valued functions.

\begin{definition}\rm   {\rm (see \cite[Definition 3.17]{Mordukhovich_2018})} {\rm Let $f:\mathbb{R}^n\rightarrow \overline{\mathbb{R}}$
		be a function with a finite value at $\bar{x}.$ For any $\bar y\in \widehat\partial f (\bar x)$, the map $\widehat\partial^2 f(\bar x,\bar y):\mathbb{R}^n\rightrightarrows \mathbb{R}^n$ with the values
		\begin{align*}
			\partial^2 f(\bar x,\bar y)(v):=(D^* \partial f)(\bar x,\bar y)(v)=\{u \mid (u,-v)\in N ( (\bar x, \bar y), \gph \partial f)\}
		\end{align*} is said to be the \textit{limiting (Mordukhovich)
			second-order subdifferential} of $f$ at $\bar x$ relative to	$\bar y.$}
\end{definition}

If $\partial f (\bar x)$ is a singleton, the symbol $\bar y$ in the notation $\partial^2 f (\bar x,\bar y)(v)$ will be omitted.  It is well-known \cite[p. 124]{Mordukhovich_2018} that if $f$ is a $C^2$-smooth function around $\bar x$, i.e., $f$ is twice continuously differentiable in a neighborhood of $\ox$, then  
\begin{align*}
	\partial^2 f(\bar x)(v)=\{\nabla^2 f (\ox)^* v\}= \{\nabla^2 f(\bar x)v\}, \  v\in \mathbb{R}^n
\end{align*}
with $\nabla^2 f (\ox)$ being the Hessian matrix of $f$ at $\ox$.

\medskip
Let $D$ be an open subset of $\mathbb{R}^n$. We denote by $C^{1,1}(D)$ the class of all real-valued functions $f$, which are \F\ differentiable on $D$, and whose gradient mapping $\nabla f (\cdot)$ is locally Lipschitz on $D.$ According to~\cite[Theorem 1.90]{Mordukhovich_2006}, if $f\in C^{1,1}(D)$ and $\bar x\in D$, one has
\begin{align}\label{ct2.1}
	\partial^2 f(\bar x) (v):= \partial^2 f (\bar x, \nabla f (\bar x))(v)=\partial \langle v, \nabla f \rangle (\bar x), \ \forall v\in \mathbb{R}^n.
\end{align}

The following properties can obtained directly from the definition.
\begin{proposition}\label{property}
	Let $f \in C^{1,1}(D)$ and $\bar x\in D$. The following assertions hold: \par 
	{\rm (i)} For any $\lambda \ge 0$, one has $\partial^2 f(\bar x)(\lambda v) = \lambda \partial^2 f(\bar x)(v), \forall v\in \mathbb{R}^n.$
	\par
	{\rm (ii)} For any $v\in \mathbb{R}^n$ the mapping $x \mapsto \partial^2 f(x) (v)$ is locally bounded. Moreover, if $x_k \to \bar x$, $x_k^* \to x^*$, $x_k^* \in \partial^2 f(x_k)(v)$ for all $k\in \mathbb{N}$, then $x^*\in \partial^2 f(\bar x)(v).$
\end{proposition}

The Taylor formula for $C^{1,1}$ functions, employing the limiting second-order subdifferential, is central to our investigation. 
\begin{theorem} {\rm (see~\cite[Theorem 3.1]{Feng_Li_2020}}\label{Taylor_formula}
	Let $f \in C^{1,1}(D)$ and $[a,b]\subseteq D$. Then, there exist $z\in \partial^2 f(\xi) (b-a)$, where $\xi \in [a,b]$, $z'\in \partial^2 f(\xi') (b-a)$, where $\xi' \in [a,b]$, such~that
	$$ \dfrac{1}{2} \langle z', b-a \rangle \le f(b)-f(a)-\langle\nabla f(a), b-a \rangle \le \dfrac{1}{2}\langle z, b-a \rangle.$$
\end{theorem}

\medskip

Consider the unconstrained problem
\begin{align}\tag{P}
	\label{unconstraint_problem}
	\min \{ f(x)\mid x \in \mathbb{R}^n\},
\end{align}
where $f\in C^{1,1}(\mathbb{R}^n).$ 

We end this section by results on the second-order necessary and sufficient optimality conditions for~\eqref{unconstraint_problem} which are formulated in~\cite{ChieuLeeYen2017}.
\begin{theorem} {\rm (see~\cite[Theorem 3.7]{ChieuLeeYen2017})}
	If $\bar x$ is a local solution of~\eqref{unconstraint_problem}, then
	\par {\rm (i)} $\nabla f(\bar x)=0,$
	\par 
	{\rm (ii)} for each $v\in \mathbb{R}^n$, there exists $z\in \partial^2 f(\bar x)(v)$ such that $\langle z, v \rangle \ge 0$.
\end{theorem}
\begin{theorem} {\rm (see~\cite[Corollary 4.8]{ChieuLeeYen2017})}
	Let $\bar x\in \mathbb{R}^n$. If the following conditions~hold
	\par {\rm (i)} $\nabla f(\bar x)=0,$
	\par 
	{\rm (ii)} for each $v\in \mathbb{R}^n\setminus\{0\}$, $\langle z, v \rangle > 0$ for all $z\in \partial^2 f(\bar x)(v)$.
	
	Then $\bar x$ is a local unique optimal solution of~\eqref{unconstraint_problem}.
	
\end{theorem}

\section{Second-order necessary optimality conditions}\label{Section-3}
In the paper, we will study a general optimization problem. More precisely, the optimization problem considered is of the form
\begin{align}\tag{P1}
	\label{constraint_problem}
	\min \{ f(x) \mid x \in \mathbb{R}^n,  g_i(x) \le 0, i\in I, h_j(x)=0, j\in J\},
\end{align}
where $f , g_i, h_j \in C^{1,1}(\mathbb{R}^n)$, $I=\{1,2,...,m\}$, $J=\{1,2,...,p\}$. 
The feasible set is denoted~by
\begin{equation}\label{constraint-set}
	X := \{x\in \mathbb{R}^n \mid g_i(x) \le 0, \ i\in I, h_j(x)=0, j\in J\}.	
\end{equation}
Given $x \in X,$ we denote the {\em active index} set to $x$ by
$$I^0(x) := \{i\in I \mid g_i(x)=0\}$$
and denote by $C(x)$ the {\em  cone of critical directions}
\begin{align*}C(x) := \{ v \in \mathbb{R}^n \mid \langle \nabla f(x),v\rangle = 0,\ &\langle \nabla g_i(x),v\rangle \le 0, \ i\in I^0(x),\\
	&\langle \nabla h_j (x), v \rangle =0, \ j\in J\}.
\end{align*}

We invoke the notion of an isolated local solution of order $2$, as introduced by Auslender~\cite[Definition 2.1]{Auslender_1984}.

\begin{definition}\rm 
	We say that a feasible point $\bar x\in X$ is an \textit{isolated local solution of order~2} of~\eqref{constraint_problem} if there exist a positive constant $\rho$ and a neighborhood $U$ of $\bar x$
	such~that
	\begin{align}\label{ils_order 2}
		f (x)> f (\bar x)  +\frac{1}{2} \rho \|x-\bar x\|^2, \ \forall x\in X \cap U, \ x\not= \bar x.
	\end{align}
\end{definition}

Observe that if we can find a positive constant $\bar\rho$ and a neighborhood $U$ of $\bar x$ with
\begin{align}\label{ils_order 2a}
	f (x)\geq f (\bar x)  +\frac{1}{2} \bar\rho  \|x-\bar x\|^2, \ \forall x\in X \cap U
\end{align} then, for any fixed $\rho \in (0,\bar\rho)$,  condition~\eqref{ils_order 2} is fulfilled. Therefore,  a feasible point $\bar x$ of~\eqref{constraint_problem} is an isolated local solution of order~$2$ if and only if there exist a positive constant $\bar \rho$ and a neighborhood $V$ of $\bar x$
such that~\eqref{ils_order 2a} holds.	

\begin{definition}\rm  {\rm (see~\cite{Khanh-et al-24})}
	One says that the \textit{metrically subregular} (MSCQ) holds at $\bar x \in X$ if there exist a neighborhood $U$ of $\bar x$ and a number $\kappa>0$ such that the  following condition is satisfied
	\begin{equation}\label{MSCQ}
		\mathrm{dist}\,(x; X)\leq \kappa\left(\sum_{i\in I}\max\{g_i(x), 0\}+\sum_{j\in J}|h_j(x)|\right)\ \ \text{for all}\ \ x\in U.
	\end{equation} 
\end{definition}
In \cite[Proposition 9.1]{Khanh-et al-24}, by using \cite[Lemma 2.5]{Chieu-Hien-17} the author proved that if the MSCQ holds at $\bar x$, then the following {\it Abadie constraint qualification} (ACQ) is satisfied
$$T( \bar x, X)=\{ v\in \mathbb{
	R}^n \mid \, \langle \nabla g_i (\bar x), v \rangle \le 0,\,  i\in I^0(\bar x), \langle \nabla h_j (\bar x),\, v \rangle =0, j\in J\},$$
where $g_i$, $i\in I$, and $h_j$, $j\in J$, are $C^2$ smooth functions. In the next proposition, we show that this result is still true for the case where functions $g_i$, $i\in I$, and $h_j$, $j\in J$, are only differentiable at the point considered.     
\begin{proposition}\label{Pro-3.3} Assume that functions $g_i$, $i\in I$, and $h_j$, $j\in J$, are  differentiable at $\bar x\in X$.   If the MSCQ holds at $\bar x$, then so does the ACQ.
\end{proposition}
\begin{proof} Let $U$ be a neighborhood of $\bar x$ and $\kappa $ be a positive number such that~\eqref{MSCQ} is satisfied.  Put 
	$$L(\bar x, X):=\{ v\in \mathbb{
		R}^n \mid \, \langle \nabla g_i (\bar x), v \rangle \le 0,\,  i\in I^0(\bar x), \langle \nabla h_j (\bar x), v \rangle =0, \, j\in J\}.$$ 
	By definition, it is easy to see that $T(\bar x, X)\subset L(\bar x, X)$. We now show that $L(\bar x, X)~\subset~T(\bar x, X)$ and so the ACQ holds at $\bar x$. Let any $v\in L(\bar x, X)$ and consider arbitrary sequences $t_k\to 0^+$, $v_k\to v$. For each $k\in\mathbb{N}$, put $x_k:=\bar x+t_kv_k$. Since $x_k\to \bar x$ as $k\to\infty$, without any loss of generality we may assume that $x_k\in U$ for all $k\in\mathbb{N}$. Let us denote 
	$$I^0(\bar x, v):=\{i\in I^0(\bar x)\,\mid\, \langle\nabla g_i(\bar x), v\rangle=0\}.$$
	We claim that 
	\begin{equation}\label{extra-equa-1}
		g_i(x_k)< 0 \ \ \forall i\in I\setminus I^0(\bar x, v)   
	\end{equation}
	and $k$ large enough. Indeed, if $i\notin I^0(\bar x)$, then $g_i(\bar x)<0$ and so $g_i(x_k)< 0$ for all $k$ large enough due to the continuity of $g_i$ and the fact that $x_k\to \bar x$ as $k\to\infty$. If $i\in I^0(\bar x)\setminus I^0(\bar x, v)$, then $g_i(\bar x)=0$ and $\langle\nabla g_i(\bar x), v\rangle<0$. By the differentiability of $g_i$ at $\bar x$, we have
	$$g_i(x_k)=g_i(x_k)-g_i(\bar x)=t_k\langle\nabla g_i(\bar x), v_k\rangle+o(t_k),$$
	with $o(t_k)/t_k\to 0$ as $k\to\infty$. Hence
	$$\lim_{k\to\infty}\frac{g_i(x_k)}{t_k}=\langle\nabla g_i(\bar x), v\rangle<0$$
	and we therefore get $g_i(x_k)<0$ for all $k$ large enough, as required. Without any loss of generality, we assume that \eqref{extra-equa-1} holds for all $k\in\mathbb{N}$. This and the MSCQ at $\bar x$ imply~that
	\begin{equation*}
		\mathrm{dist}\,(x_k, X)\leq \kappa\left(\sum_{i\in I^0(\bar x, v)}\max\{g_i(x_k), 0\}+\sum_{j\in J}|h_j(x_k)|\right)\ \ \forall k\in\mathbb{N}.
	\end{equation*}
	By the differentiability of $g_i$ and $h_j$ at $\bar x$, one has
	\begin{align*}
		g_i(x_k)&= t_k\langle\nabla g_i(\bar x), v_k\rangle+o(t_k), \ \ i\in I^0(\bar x,v),
		\\
		h_j(x_k)&=t_k\langle \nabla h_j(\bar x), v_k\rangle +o(t_k),\ \ j\in J.
	\end{align*}
	This implies that 
	\begin{align*}
		\lim_{k\to\infty}\frac{g_i(x_k)}{t_k}&= \langle\nabla g_i(\bar x), v\rangle=0, \ \ \forall i\in I^0(\bar x,v),
		\\
		\lim_{k\to\infty}\frac{h_j(x_k)}{t_k}&=\langle \nabla h_j(\bar x), v\rangle=0,\ \ \forall j\in J.
	\end{align*}
	Hence
	\begin{equation*}
		\lim_{k\to\infty}\frac{\mathrm{dist}\,(x_k, X)}{t_k}=0.
	\end{equation*}
	This and the closedness of $X$ imply that there exists a sequence $z_k$ in $X$ such~that
	\begin{equation*}
		\lim_{k\to\infty}\frac{\|z_k-x_k\|}{t_k}=0.
	\end{equation*}
	For each $k\in\mathbb{N}$, put $w_k:=(z_k-x_k)/t_k$, Then $w_k\to 0$ and we have 
	$$z_k=x_k+t_kw_k=\bar x+t_k(v_k+w_k).$$
	Since $z_k\in X$, $t_k\to 0^+$, and $v_k+w_k\to v$, one has $v\in T(\bar x, X)$. Thus, we have $L(\bar x, X)\subset T(\bar x, X)$, as required. 
\end{proof}
The Lagrangian associated with the constrained problem~\eqref{constraint_problem} is given by
\begin{align*}
	\mathcal{L} (x, \lambda, \mu)=f(x)+\sum\limits_{i\in I} \lambda_i g_i (x)+ \sum\limits_{j\in J} \mu_j h_j(x).
\end{align*}
\begin{definition}\rm Let $\bar x\in X$. The set of \textit{Lagrange multipliers} of problem~\eqref{constraint_problem} at $\bar x$ is defined by
	$$\Lambda (\bar x):=\{(\bar\lambda,\bar\mu)\in\mathbb{R}^m_+\times\mathbb{R}^p\,\mid\, \nabla_x\mathcal{L}(\bar x, \bar\lambda, \bar\mu)=0, \bar \lambda_i g_i(\bar x)=0, \ i\in I\}.$$
	If $\Lambda (\bar x)\neq \emptyset$, then $\bar x$ is called a {\em stationary point} of \eqref{constraint_problem}.  	
\end{definition}

The following theorem provides second-order necessary optimality conditions for problem~\eqref{constraint_problem} in terms of limiting second-order subdifferentials.
{ 
	\begin{theorem}\label{main_theorem}
		Suppose that $\bar x\in X$ is a local solution of~\eqref{constraint_problem} and  the MSCQ holds at $\bar x$.  Then, the following assertions hold:
		\par {\rm (i)} $\langle\nabla f(\bar x), v \rangle \ge 0$ for any $v\in T(\bar x, X).$ Consequently, $\Lambda (\bar x)$ is nonempty. 
		\par {\rm (ii)}	Let $(\bar\lambda, \bar\mu)\in \Lambda(\bar x)$. Then for each $v\in T(\bar x, X^0)$, there exists $z\in\partial^2 \mathcal{L}(\bar x, \bar\lambda, \bar \mu) (v)$ such~that 
		\begin{equation}\label{equa-2-new}
			\langle z, v \rangle \ge 0,
		\end{equation}
		where 
		\begin{align*}X^0=\{x\in \mathbb{R}^n \mid h_j(x)=0, \ j\in J, \  &g_i(x)=0 \ \mbox{if} \ i\in I^0(\bar x) \ \mbox{and} \ \bar\lambda_i>0, 
			\\
			&\quad\quad\quad\quad\quad\quad g_i(x)\le 0, \ \mbox{otherwise}\}.
		\end{align*}
		\par {\rm (iii)} If $\bar x$ is an isolated local solution of order $2$  of \eqref{constraint_problem} and $(\bar\lambda, \bar\mu)\in \Lambda(\bar x)$, then,	for every $v\in T(\bar x, X^0)\setminus\{0\}$, there exists $z\in \partial^2 \mathcal{L}(\bar x, \bar\lambda, \bar \mu) (v)$ such~that $$\langle z, v \rangle >0.$$
	\end{theorem}
	\begin{proof}
		\par {\rm (i)} Let $v\in T(\bar x, X)$. Then, by definition, there exist  a sequence of positive scalars $t_k \to 0^+$ and a sequence of vectors $v_k\to v$  such that 	$x_k:= \bar x +t_k v_k\in X$ for all $k\in\mathbb{N}$. Since $\bar x$ is a local solution of \eqref{constraint_problem} and $x_k\to \bar x$ as $k\to\infty$, we have $f(x_k)\geq f(\bar x)$ for all $k$ large enough. By the differentiability of $f$ at $\bar x$, one has
		\begin{equation*}
			0\leq f(x_k)-f(\bar x)=t_k\langle\nabla f(\bar x), v_k\rangle+ o(t_k),
		\end{equation*}
		with $o(t_k)/t_k\to 0$ as $k\to\infty$. Dividing two sides by $t_k$ and letting $k\to\infty$ we obtain $\langle\nabla f(\bar x), v \rangle \ge 0$. Now by the MSCQ holds at $\bar x$, one has $\langle\nabla f(\bar x), v \rangle \ge 0$ for all $v$ satisfies
		$$\langle \nabla g_i (\bar x), v \rangle \le 0,\  i\in I^0(\bar x), \langle \nabla h_j (\bar x), v \rangle =0,\ \ j\in J.$$
		This means that the following system 
		$$\langle\nabla f(\bar x), v \rangle<  0, \langle\nabla g_i (\bar x), v \rangle \le 0,\  i\in I^0(\bar x), \langle \nabla h_j (\bar x), v \rangle =0,\ \ j\in J$$
		has no solution $v\in\mathbb{R}^n$. Hence, by the Motzkin's theorem of the alternative (see \cite[pp. 28--29]{Mangasarian}), there exist  $\bar\lambda_i\geq 0$, $i\in I^0(\bar x)$, and $\bar\mu_j\in\mathbb{R}$, $j\in J$, such~that  
		$$\nabla f(\bar x)+\sum_{i\in I^0(\bar x)}\bar \lambda_i\nabla g_i(\bar x)+\sum_{j\in J}\bar\mu_j\nabla h_j(\bar x)=0.$$
		For $i\in I\setminus I^0(\bar x)$, put $\bar\lambda_i=0$ and $\bar\lambda:=(\bar\lambda_i)_{i\in I}$, $\bar\mu:=(\mu_j)_{j\in J}$. Then $(\bar\lambda,\bar\mu)\in \Lambda(\bar x)$, as required.
		
		\par (ii) Suppose that $\bar x$ is a local minimum of problem \eqref{constraint_problem} and $(\bar\lambda, \bar\mu)\in \Lambda(\bar x)$. Let any $v\in T(\bar x, X^0)$. Since $\bar x$ is a local minimum of~\eqref{constraint_problem}, there exists a  neighborhood $U$ of $\bar x$ such that 
		$$f(x)-f(\bar x) \ge 0, \ \forall x\in X\cap U.$$
		As $v\in T(\bar x, X^0)$, there exist  a sequence of positive scalars $t_k \to 0^+$ and a sequence of vectors $v_k\to v$  such that 
		$x_k:= \bar x +t_k v_k$ 		belongs to $ X^0 \cap U$ for all $k\in \mathbb{N}.$ Since $ x_k\in X^0$, we have $\bar{\lambda_i} g_i(x_k)=0$ for $i\in I^0(\bar x)$ and $h_j(x_k)=0$ for $j\in J$. Hence
		\begin{equation*}
			\mathcal{L} (x_k, \bar \lambda, \bar \mu)- \mathcal{L} (\bar x, \bar \lambda, \bar \mu)= f(x_k)-f(\bar x) \ge 0.
		\end{equation*}
		Combining this with the fact that $\nabla_x \mathcal{L}(\bar x, \bar \lambda, \bar \mu)=0$ one has
		\begin{align}\label{Taylor_expres_2n} \nonumber
			0&\le 	\mathcal{L} (x_k, \bar \lambda, \bar \mu)- \mathcal{L} (\bar x, \bar \lambda, \bar \mu)\\ \nonumber
			&= [ \mathcal{L} (x_k, \bar \lambda, \bar \mu) -\mathcal{L} (\bar x + t_k v, \bar \lambda, \bar \mu)] +[\mathcal{L} (\bar x + t_k v, \bar \lambda, \bar \mu)- \mathcal{L} (\bar x, \bar \lambda, \bar \mu) ]\\ \nonumber
			& = [ \mathcal{L} (x_k, \bar \lambda, \bar \mu) -\mathcal{L} (\bar x + t_k v, \bar \lambda, \bar \mu)] +[\mathcal{L} (\bar x + t_k v, \bar \lambda, \bar \mu)- \mathcal{L} (\bar x, \bar \lambda, \bar \mu) \\
			&  \quad \quad \quad\quad \quad \quad \quad \quad \quad \quad \quad \quad \quad \quad \quad \quad \quad \quad  \quad  -  \langle \nabla_x \mathcal{L}(\bar x, \bar \lambda, \bar \mu) , t_k v \rangle].
		\end{align} 
		Setting 
		$$A= \mathcal{L} (\bar x + t_k v, \bar \lambda, \bar \mu)- \mathcal{L} (\bar x, \bar \lambda, \bar \mu) -  \langle \nabla_x \mathcal{L}(\bar x, \bar \lambda, \bar \mu) , t_k v \rangle $$
		and 
		$$B= \mathcal{L} (x_k, \bar \lambda, \bar \mu) -\mathcal{L} (\bar x + t_k v, \bar \lambda, \bar \mu).$$
		By applying the Taylor formula (Theorem~\ref{Taylor_formula}) we can find $z_k\in \partial^2 \mathcal{L} (\xi_k, \bar \lambda, \bar \mu) (t_k v)$ such that 
		\begin{align}\label{Taylor_expres_3}
			A \le  \dfrac{1}{2} \langle z_k, t_k v\rangle,
		\end{align}
		where
		$	\xi_k \in [\bar x, \bar x+t_k v].$
		From Proposition~\ref{property} (i),  $z_k\in \partial^2 \mathcal{L} (\xi_k, \bar \lambda, \bar \mu) (t_k v)$ means that $z_k\in t_k \partial^2 \mathcal{L} (\xi_k, \bar \lambda, \bar \mu) ( v).$
		The latter allows us to express $z_k= t_k  w_k$, where $w_k\in \partial^2 \mathcal{L} (\xi_k, \bar \lambda, \bar \mu) ( v).$ Thus,~\eqref{Taylor_expres_3} can be rewritten as
		\begin{align}\label{Taylor_expres_4}
			A \le  \dfrac{1}{2}t_k^2 \langle w_k, v\rangle.
		\end{align}
		For each $k$, the function $\varphi (t):= \mathcal L ((1-t)\bar x +t x_k, \bar \lambda, \bar \mu)$, where $t\in [0,1]$ is the composition function of the affine function $t\mapsto (1-t) \bar x +tx_k$ defined on $[0,1]$ and the vector function $x\mapsto \mathcal{L} (x, \bar \lambda, \bar \mu)$ defined on $\mathbb{R}^n$. Under our assumptions, we have $\mathcal{L} (\cdot, \bar \lambda, \bar \mu) \in C^{1,1} (\mathbb{R}^n)$. By using the classical mean value theorem (see, e.g., \cite[Theorem~5.10]{Rudin_PMA1976}) to $\varphi$ on $[0,1]$ and using the chain rule (see, e.g., \cite[p.~103]{Lang_1983}), we find $\zeta_k\in (\bar x+t_k v, x_k)$ such that
		\begin{equation*}
			B=	\mathcal{L}(x_k,\bar\lambda, \bar \mu )  -\mathcal{L}(\bar x+t_k v, \bar\lambda, \bar \mu) = \langle \nabla_x \mathcal{L}(\zeta_k,\bar\lambda, \bar \mu ), x_k -(\bar x +t_k v) \rangle.
		\end{equation*}
		The later is equivalent to
		\begin{equation*}
			B= \langle \nabla_x \mathcal{L}(\zeta_k,\bar\lambda, \bar \mu ), t_k(v_k -v) \rangle,
		\end{equation*}
		and, hence
		\begin{align}\label{B_express} \nonumber
			|B|&= |\langle \nabla_x \mathcal{L}(\zeta_k,\bar\lambda, \bar \mu ), t_k(v_k -v) \rangle| \\ \nonumber
			&= |\langle \nabla_x \mathcal{L}(\zeta_k,\bar\lambda, \bar \mu ) -   \nabla_x \mathcal{L}(\bar x,\bar\lambda, \bar \mu ), t_k(v_k -v) \rangle| \\ 
			& \le \|  \nabla_x \mathcal{L}(\zeta_k,\bar\lambda, \bar \mu )-   \nabla_x \mathcal{L}(\bar x,\bar\lambda, \bar \mu )\|. \| t_k (v_k-v)\|.
		\end{align}
		Since $\mathcal{L} (\cdot, \bar \lambda, \bar \mu) \in C^{1,1} (\mathbb{R}^n)$, we can find a positive constant $\ell$ such that 
		\begin{align}\label{Lipschit_L}\nonumber
			\| \nabla_x \mathcal{L}(\zeta_k,\bar\lambda, \bar \mu )- \nabla_x \mathcal{L}(\bar x,\bar\lambda, \bar \mu )\| &\le \ell \| \zeta_k -\bar x\| \\ \nonumber
			& \le \ell [\|\zeta_k -x_k \| +\|x_k-\bar x\|]\\
			& = \ell t_k[\|v_k-v\| +\|v_k\|], \ \ k\in \mathbb{N}.
		\end{align}
		Combining~\eqref{B_express} and~\eqref{Lipschit_L} yields
		\begin{equation*} 
			|B|\le \ell t_k^2[\|v_k-v\| +\|v_k\|]. \|v_k-v\|.
		\end{equation*}
		This implies that 
		\begin{equation}\label{B_express_1}
			\dfrac{|B|}{t_k^2} \to 0 \ \mbox{as}  \ k\to \infty.
		\end{equation}
		From~\eqref{Taylor_expres_2n} and~\eqref{Taylor_expres_4} we get
		\begin{align*}
			0\le A +B \le  \dfrac{1}{2}t_k^2 \langle w_k, v\rangle+B.
		\end{align*}
		Dividing both sides by $t_k^2$ we obtain
		\begin{align*}
			0\le  \dfrac{1}{2} \langle w_k, v\rangle+\dfrac{B}{t_k^2}.
		\end{align*}
		Since $\partial^2 \mathcal L (., \bar \lambda, \bar \mu)(v) $ is locally bounded at $\bar x$, and $\xi_k \to \bar x$, we deduce that $\{w_k\}$ is bounded. Thus, we can assume, without any loss of generality, that $w_k$ converges to some $z\in\partial^2 \mathcal{L} (\bar x, \bar \lambda, \bar \mu)$, and, hence
		$$0\le \lim\limits_{k\to \infty}  \dfrac{1}{2} \langle w_k, v\rangle+\dfrac{B}{t_k^2} = \dfrac{1}{2} \langle z, v\rangle$$
		due to~\eqref{B_express_1}. In other words, we can find $z\in\partial^2 \mathcal{L} (\bar x, \bar \lambda, \bar \mu)$ such that
		$$\langle z, v \rangle  \ge 0.$$
		
		(iii) Suppose that $\bar x$ is an isolated local solution of order $2$ of \eqref{constraint_problem} and $(\bar\lambda, \bar\mu)\in \Lambda(\bar x)$. Let $\bar\rho$ be a positive number  and $U$ be a neighborhood  of $\bar x$ such that  
		\begin{equation*} 
			f(x)-f(\bar x)\geq \bar\rho\|x-\bar x\|^2, \ \ \forall x\in X\cap U.
		\end{equation*}
		For $v\in T(\bar x, X^0)\setminus\{0\}$, there exist  a sequence of positive scalars $t_k \to 0^+$ and a sequence of vectors $v_k\to v$  such that 
		$x_k:= \bar x +t_k v_k$ 		belongs to $ X^0 \cap U$ for all $k\in \mathbb{N}.$ Hence
		\begin{equation*} 
			\mathcal{L} (x_k, \bar \lambda, \bar \mu)- \mathcal{L} (\bar x, \bar \lambda, \bar \mu)= f(x_k)-f(\bar x) \ge \bar\rho t_k^2\|v_k\|^2.
		\end{equation*}
		By this and a similar argument as in the proof of part (ii), we see that there exists $z\in\partial^2\mathcal{L}(\bar x, \bar\lambda, \bar\mu)$ such that
		$$\frac{1}{2}\langle z, v\rangle\geq \bar\rho\|v\|^2>0,$$  
		as required.    
	\end{proof}
}
\begin{remark}\rm 
	Feng and Li \cite{Feng_Li_2020} investigated second-order necessary conditions for optimization problems involving set constraints and finite inequality constraints, employing the concept of limiting second-order subdifferentials. Our approach differs significantly from that of \cite{Feng_Li_2020}.
	Firstly, we focus on optimization problems with finite inequality and equality constraints. Secondly, we employ the Lagrangian method and utilize information from the first-order tangent set, whereas the optimality conditions in \cite{Feng_Li_2020} rely on information from the second-order tangent set.
	Especially in Theorem~\ref{main_theorem}, we study the second-order necessary optimality conditions for the isolated solution of order 2 as well.	
\end{remark}

\section{Second-order sufficient  optimality conditions}\label{Section-4}
Let us now turn our attention to second-order sufficient optimality conditions.

\begin{theorem}\label{sufficient-theorem-1} Let $\bar x$ be a stationary point of~\eqref{constraint_problem} and $(\bar\lambda, \bar \mu) \in\Lambda(\bar x) $. Assume that for every $v\in [T(\bar x, X)\cap C(\bar x)]\setminus \{0\}$  and for any $z\in \partial^2 \mathcal{L}(\bar x,\bar \lambda, \bar \mu ) (v) $ one~has 
	\begin{align*}
		\big	\langle z, v \big\rangle > 0.
	\end{align*}
	Then $\bar x$ is an isolated local solution of order $2$ of~\eqref{constraint_problem}.
\end{theorem}
\begin{proof}
	We argue by contradiction. Suppose that $\bar x$ is not an isolated local solution of order $2$ of~\eqref{constraint_problem}. Then there exists a sequence $x_k\in X$ such that $x_k\to \bar x$ $(x_k\not= \bar x)$~and 
	\begin{align}\label{xk}
		f(x_k) \leq f(\bar x)+\frac{1}{k}t_k^2, \ \forall k,
	\end{align}
	where $t_k:=\| x_k-\bar x\| >0$ and $t_k\to 0$ as $k\to\infty$. Set $v_k=\frac{x_k-\bar x}{t_k}.$ Since $v_k$ is bounded, there exists a subsequence, denoted in the same way, such that $v_k \to v$ as $k\to \infty$ and $\|v\|=1$. Clearly, $v\in T(\bar x, X)$.
	
	For $i\in I^0(\bar x)$, we can expand $g_i(x)$ around $\bar x$ as follows
	$$0\ge g_i(x_k)=\langle \nabla g_i(\bar x), x_k-\bar x \rangle + o(x_k-\bar x),$$
	with $o((x_k-\bar x)/{\|x_k-\bar x\|}) \to 0$ as $x_k\to \bar x.$ Dividing both sides by $\|x_k-\bar x\|$ and passing to the limit over the subsequence for which $v_k\to v$, one gets
	$$\langle \nabla g_i(\bar x), v \rangle \le 0, \ i\in I^0(\bar x).$$
	A similar analysis of equality constraints yields the relations
	$$\langle \nabla h_j(\bar x), v \rangle =0, \ j\in J.$$
	Since $\bar x$ is a stationary point, we have  $$\langle \nabla f(\bar x), v \rangle= -\sum_{i\in I^0(\bar x)}\bar{\lambda}_i\langle\nabla g_i(\bar x), v\rangle-\sum_{j\in J}\bar{\mu}_j\langle\nabla h_j(\bar x), v\rangle \ge 0.$$ 
	On the other hand, from~\eqref{xk} we have
	\begin{align*}
		\langle \nabla f(\bar x), x_k-\bar x \rangle +o(x_k-\bar x)=f(x_k)-f(\bar x) \leq \frac{1}{k}t_k^2.
	\end{align*}
	Dividing both sides by $\|x_k-\bar x\|$ and passing to the limit $x_k\to \bar x$ we get $ \langle \nabla f(\bar x), v \rangle \le 0.$
	Consequently, $ \langle \nabla f(\bar x), v \rangle = 0.$ Therefore $v\in C(\bar x).$
	
	By Taylor formula (Theorem~\ref{Taylor_formula}), we can find $z_k\in \partial^2 \mathcal{L}(\xi_k,\bar \lambda, \bar \mu ) (x_k-\bar x)$, where $\xi_k\in [x_k,\bar x]$ such that
	\begin{align*}
		\dfrac{1}{2} \langle z_k, x_k-\bar x\rangle \le \mathcal{L}(x_k, \bar \lambda, \bar \mu) - \mathcal{L}(\bar x, \bar \lambda, \bar \mu) - \langle \nabla_x\mathcal{L}(\bar x, \bar \lambda, \bar \mu) , x_k-\bar x \rangle.
	\end{align*}
	This and the fact that $x_k\in X$ imply that   
	\begin{equation*}\label{formula_1}
		\aligned
		\dfrac{1}{2} &\langle z_k, t_kv_k\rangle \le \mathcal{L}(x_k, \bar \lambda, \bar \mu) - \mathcal{L}(\bar x, \bar \lambda, \bar \mu)
		\\
		&=[f(x_k)-f(\bar x)]+\sum_{i\in I^0(\bar x)}\bar\lambda_i [g_i(x_k)-g_i(\bar x)]+\sum_{j\in J}\bar{\mu}_j[h_j(x_k)-h_j(\bar x)]\leq \frac{1}{k}t_k^2.
		\endaligned
	\end{equation*}
	By Proposition~\ref{property}(i) and the fact that $z_k\in \partial^2 \mathcal{L} (\xi_k, \bar \lambda, \bar \mu) (t_k v_k)$, there exists $z'_k~\in~\partial^2 \mathcal{L}(\xi_k,\bar \lambda, \bar \mu ) (v_k)$ such  that  $z_k=t_k z'_k$.
	Hence
	$$\dfrac{1}{2}t_k^2 \langle z'_k, v_k\rangle \leq   \frac{1}{k}t_k^2,\ \ \forall k\in\mathbb{N},$$ 
	and it implies that $\langle z, v \rangle \le 0$ for some $z\in \partial^2 \mathcal{L}(x,\bar \lambda, \bar \mu ) (v) $, which contradicts the assumptions. 
\end{proof}

Let us see an example to illustrate Theorems \ref{main_theorem} and \ref{sufficient-theorem-1}.
\begin{example} {\rm  Consider the problem~\eqref{constraint_problem} with $m=2$, $p=1$,
		$$f(x_1,x_2)=-\int_0^{x_1} |t|dt + x_2^2,$$ $g_1(x_1,x_2)=-x_1$,  $g_2(x_1,x_2)=-x_2$, and $h(x_1,x_2)=x_1+x_2-1.$  	It is easy to check that $f \in C^{1,1}(\mathbb{R}^2)$ and $g_1,g_2,h\in C^2$. 
		
		The Lagrangian function of problem~\eqref{constraint_problem} is
		$$\mathcal{L}(x, \lambda, \mu)= -\int_0^{x_1} |t|dt + x_2^2 -\lambda_1x_1 -\lambda_2 x_2 + \mu (x_1+x_2-1).$$ 
		Then, we have $\nabla_x \mathcal{L}(x, \lambda,\mu)=( -|x_1|-\lambda_1 +\mu, 2x_2 -\lambda_2 +\mu).$  Hence, $\bar x=(\bar x_1, \bar x_2)\in\mathbb{R}^2$ is a stationary point of \eqref{constraint_problem} with respect to Lagrange multipliers $\bar\lambda=(\bar\lambda_1, \bar\lambda_2)$, $\bar\mu$ if the following conditions hold
		\begin{equation*}\label{equa-1-new}
			\begin{cases}
				-|\bar x_1|-\bar \lambda_1 +\bar\mu=0,
				\\
				2\bar x_2 -\bar\lambda_2 +\bar \mu=0,
				\\
				\bar x_1\geq 0, \bar x_2\geq 0, \bar x_1+\bar x_2=1,
				\\
				\bar\lambda_1\geq 0, \bar\lambda_2\geq 0, 
				\\
				\bar\lambda_1\bar x_1=\bar\lambda_2\bar x_2=0. 
			\end{cases}
		\end{equation*}  
		It is easy to see that the above system has a unique solution $\bar x=(1,0)$ with respect to $\bar \lambda=(0,1)$ and $\bar \mu=1$. Then, $X^0=\{\bar x\}$. Hence, $T(\bar x, X^0)=\{0\}$ and so the second-order necessary optimality condition \eqref{equa-2-new} is satisfied.  By Theorem \ref{main_theorem}, the problem \eqref{constraint_problem} has only one candidate  for optimal solutions that is  $\bar x=(1,0)$. 
		
		We now use Theorem \ref{sufficient-theorem-1} to show that $\bar x=(1,0)$ is an isolated local solution of order~$2$ of \eqref{constraint_problem}. An easy computation shows that  $C(\bar x)=\{0\}$. Hence, $[T(\bar x, X)~\cap~C(\bar x)]\setminus \{0\}$  is empty. This and Theorem \ref{sufficient-theorem-1} imply that $\bar x=(1,0)$ is an isolated local solution of order $2$ of \eqref{constraint_problem}.     
	}
\end{example}

\begin{remark}\rm 
	Theorem~\ref{sufficient-theorem-1} establishes a sufficient optimality condition that differs from the one presented in \cite[Theorem~4.2]{Feng_Li_2020}. Our condition requires that the second-order subdifferential is positive definite for all directions $z$ belonging to the set $\partial^2 \mathcal{L}(\bar x,\bar \lambda, \bar \mu ) (v)$, which is a smaller set compared to the one considered in~\cite[Theorem~4.2]{Feng_Li_2020}. 
\end{remark}
We present a counterexample to demonstrate that our theorem is applicable in situations where the result in \cite[Theorem 4.2]{Feng_Li_2020} cannot be applied.

\begin{example}\rm(Theorem~\ref{sufficient-theorem-1} works but \cite[Theorem 4.2]{Feng_Li_2020} fails)
	Consider the problem~\ref{constraint_problem} with the following data:
	\begin{align*}
		f (x_1, x_2) := x_1 + x_2^2, \ g_1(x_1, x_2) := -x_1, \ g_2(x_1, x_2) := -x_2,\\
		X:=\{x=(x_1, x_2)\in\mathbb{R}^2\,|\, g_1(x_1, x_2)\leq 0, g_2(x_1, x_2)\leq 0 \},
	\end{align*}
	and $\bar x=(0,0)\in X$. First, we observe that $\bar x$ is a stationary point of~\ref{constraint_problem} with respect to a unique Lagrange multiplier $\bar \lambda =(1,0)$. By a simple calculation, we find that
	$$T(\bar x, X)=\{v \in \mathbb{R}^2 \mid \langle \nabla g_i(\bar x), v \rangle \le 0, \forall i \in I^0(\bar x)\}= \mathbb{R}^2_+$$
	and $C(\bar x)=\{0\} \times \mathbb{R}_+$. Then $v=(v_1,v_2)\in [T(\bar x, X)\cap C(\bar x)]\setminus \{0\}$ means that $v_1=0$ and $v_2 > 0$.
	
	The Lagrangian associated with~\ref{constraint_problem} is
	$$\mathcal{L}(x, \bar\lambda)= x_1 +x_2^2 -\bar\lambda_1 x_1-\bar\lambda_2 x_2=x_2^2.$$
	We have $\nabla_x \mathcal{L}(\bar x, \bar\lambda)=(0, 0)$ and $\partial^2 \mathcal{L}(\bar x,\bar \lambda) (v)=(0, 2 v_2).$ So, for any $z\in \partial^2 \mathcal{L}(\bar x,\bar \lambda) (v)$ and $v=(v_1,v_2)\in [T(\bar x, X)\cap C(\bar x)]\setminus \{0\}$, one has $\langle z, v \rangle = 2 v_2^2 > 0$. Thus, by Theorem~\ref{sufficient-theorem-1}, one obtains  $\bar x=(0,0)$ is an isolated local solution of order $2$ of~\ref{constraint_problem}.
	However, as shown in \cite[Example 4.5]{Feng_Li_2020}, Theorem 4.2 in  \cite{Feng_Li_2020} cannot be applicable in this context. 
\end{example}

%\section{Second-order Optimality Conditions for Tilt-stable Minimizers}
%%%%%%%%%%%%%%%%
\section{Second-order sufficient  optimality conditions in multiobjective optimization problems}\label{Section-5}
In this section, we consider the following multiobjective optimization problem
\begin{equation}\label{MOP}
	\mathrm{Min}\,_{\mathbb{R}^q_+}\{\varphi(x):=(\varphi_1(x), \ldots, \varphi_q(x))\,|\, x\in X\},\tag{MP}
\end{equation}
where $\varphi_l\in C^{1,1}(\mathbb{R}^n)$, $l\in L:=\{1, \ldots, q\}$, and $X$ is defined as in \eqref{constraint-set}. 

The Lagrangian associated with the constrained problem~\eqref{MOP} is given by
\begin{align*}
	\mathcal{L} (x, \alpha, \lambda, \mu)=\sum_{l\in L}\alpha_l\varphi_l(x)+\sum\limits_{i\in I} \lambda_i g_i (x)+ \sum\limits_{j\in J} \mu_j h_j(x).
\end{align*}

We say that $v\in\mathbb{R}^n$ is a {\em  critical direction} of \eqref{MOP} at a  $x\in X$ if 
\begin{equation*}
	\begin{cases}
		\langle \nabla\varphi_l(x), v\rangle\leq 0 \ \ &\forall l\in L,
		\\
		\langle \nabla\varphi_l(x), v\rangle= 0 \ \ &\text{at least one}\ \ l\in L,
		\\
		\langle \nabla g_i(x), v\rangle\leq 0 \ \ &\forall i\in I^0(x),
		\\
		\langle \nabla h_j(x), v\rangle = 0 \ \ &\forall j\in j.
	\end{cases}
\end{equation*}
The set of all critical direction of   \eqref{MOP} at $x$ is denoted by $K(x)$.
\begin{definition}\rm  Let $\bar x\in X$. The set of {\em Lagrange multipliers} of \eqref{MOP} at $\bar x$ is defined~by
	\begin{align*}
		\mathrm{A}(\bar x):=\{(\bar\alpha, \bar\lambda, \bar\mu)\in (\mathbb{R}^q_+\setminus\{0\})\times\mathbb{R}^m_+\times\mathbb{R}^p\,&|\, \nabla_x 	\mathcal{L} (\bar x, \bar\alpha, \bar\lambda, \bar\mu)=0,
		\\
		&\ \ \ \ \ \ \bar\lambda_i g_i(\bar x)=0, i\in I\}.
	\end{align*}	
\end{definition}

If the set of Lagrange multipliers of \eqref{MOP} at $\bar x$ is nonempty, then $\bar x$ is called a stationary point of \eqref{MOP}.  

\begin{definition}\rm {\rm (see, e.g.~\cite{Ginchev_etal})} Let $\bar x\in X$. We say that:
	\begin{enumerate}[(i)]
		\item $\bar x$ is a {\em local weak efficient solution} of \eqref{MOP} if there exists a neighborhood $U$ of $\bar x$ such that  
		\begin{equation*}
			\max\{\varphi_1(x)-\varphi_1(\bar x), \ldots, \varphi_q(x)-\varphi_q(\bar x)\}\geq 0 \ \ \forall x\in X\cap U. 
		\end{equation*}
		\item  $\bar x$  is an {\em  isolated local solution of order $2$} of \eqref{MOP} if there exist a positive number $\beta$ and a neighborhood $U$ of $\bar x$ such that
		\begin{equation*}
			\max\{\varphi_1(x)-\varphi_1(\bar x), \ldots, \varphi_q(x)-\varphi_q(\bar x)\}\geq \beta \|x-\bar x\|^2 \ \ \forall x\in X\cap U.
		\end{equation*}
	\end{enumerate}
\end{definition}

The following result gives a sufficient for the nonemptiness of the set of Lagrange multipliers at a given local weak efficient solution of \eqref{MOP}.
\begin{proposition} Suppose that $\bar x\in X$ is a local weak efficient solution of \eqref{MOP}. If  the MSCQ holds at $\bar x$, then   $$\max\{\langle\nabla \varphi_l(\bar x), v \rangle\,|\, l\in L\} \ge 0$$ 
	for every $v\in T(\bar x, X).$ Consequently, $\mathrm{A} (\bar x)$ is nonempty. 
\end{proposition}
\begin{proof}
	Suppose on the contrary that there exists a vector $v~\in~T(\bar x, X)$ such that   $\max\{\langle\nabla \varphi_l(\bar x), v \rangle\,|\, l\in L\} < 0,$ or, equivalently, $\langle\nabla \varphi_l(\bar x), v \rangle<0$ for all $l\in L$.  Let $t_k\to 0^+$ and $v_k\to v$ such that $\bar x+t_kv_k\in X$ for all $k\in\mathbb{N}$. For each $l\in L$, it follows~from
	\begin{equation*}
		\lim_{k\to\infty}\frac{\varphi_l(\bar x+t_kv_k)-\varphi_l(\bar x)}{t_k}=\langle\nabla \varphi_l(\bar x), v \rangle<0
	\end{equation*}
	that there exists $k_l\in\mathbb{N}$ such that
	\begin{equation*}
		\varphi_l(\bar x+t_kv_k)<\varphi_l(\bar x) \ \ \forall k\geq k_l.
	\end{equation*}   	
	Let $k_0:=\max\{k_l\,|\, l\in L\}$. Then we have
	\begin{equation*}
		\varphi_l(\bar x+t_kv_k)<\varphi_l(\bar x) \ \ \forall l\in L, k\geq k_0, 
	\end{equation*} 
	contrary to the fact that $\bar x$ is a local efficient solution of \eqref{MOP}. Hence 
	$$\max\{\langle\nabla \varphi_l(\bar x), v \rangle\,|\, l\in L\} \ge 0 \ \ \forall v\in T(\bar x, X).$$
	From this, the MSCQ condition at $\bar x$, and Proposition \ref{Pro-3.3}, we see that the following system has no solution $v$
	\begin{equation*}
		\begin{cases}
			\langle \nabla\varphi_l(\bar x), v\rangle< 0 \ \ &\forall l\in L,
			\\
			\langle \nabla g_i(\bar x), v\rangle\leq 0 \ \ &\forall i\in I^0(\bar x),
			\\
			\langle \nabla h_j(\bar x), v\rangle = 0 \ \ &\forall j\in j.
		\end{cases}
	\end{equation*}
	By the Motzkin's theorem of the alternative (see \cite[pp. 28--29]{Mangasarian}), $\mathrm{A}(\bar x)$ is nonempty. The proof is complete.
\end{proof}

The following theorem presents a second-order sufficient optimality condition for an isolated local solution of order $2$ of \eqref{MOP}.
\begin{theorem}\label{suffi_MP} Suppose that $\bar x$ is a stationary point of~\eqref{MOP} and $(\bar\alpha, \bar\lambda, \bar\mu)\in\mathrm{A}(\bar x)$. Assume that for every $v\in [T(\bar x, X)\cap K(x)]\setminus\{0\}$ and for any $z\in\partial^2\mathcal{L}(\bar x, \bar\alpha, \bar\lambda, \bar\mu)$  
	one~has
	\begin{equation*}
		\langle z, v\rangle>0.
	\end{equation*}
	Then $\bar x$ is an isolated local solution of order $2$ of \eqref{MOP}. 
\end{theorem}
\begin{proof}
	Suppose on the contrary that $\bar x$ is not an isolated local solution of order $2$ of~\eqref{MOP}. Then there exists a sequence $x_k\in X$ such that $x_k\to \bar x$ $(x_k\not= \bar x)$ and 
	\begin{align}\label{xk-MOP}
		\max\{\varphi_l(x_k)-\varphi_l(\bar x)\,|\, l\in L\}\leq \frac{1}{k}t_k^2, \ \forall k\in \mathbb{N},
	\end{align}
	where $t_k:=\| x_k-\bar x\| >0$. Clearly, $t_k\to 0$ as $k\to\infty$. Set $v_k:=\frac{x_k-\bar x}{t_k}.$ Since $v_k$ is bounded, there exists a subsequence, denoted in the same way, such that $v_k \to v$ as $k\to \infty$ and $\|v\|=1$. By  a similar argument as in the proof of Theorem \ref{sufficient-theorem-1}, we see~ that 
	\begin{align*}
		&v\in T(\bar x, X), \langle \nabla g_i(\bar x), v \rangle \le 0, \ \ i\in I^0(\bar x), \ \ \text{and}
		\\
		&\langle \nabla h_j(\bar x), v \rangle =0, \ \ j\in J.
	\end{align*}
	It follows from \eqref{xk-MOP} that
	\begin{equation*}
		\varphi_l(x_k)-\varphi_l(\bar x)\leq \frac{1}{k}t_k^2, \ \ \forall l\in L, k\in \mathbb{N}.
	\end{equation*} 
	By Taylor  formula, one has
	\begin{equation*}
		\varphi_l(x_k)-\varphi_l(\bar x)=t_k\langle\nabla\varphi_l(\bar x), v_k\rangle +o(t_k)\leq \frac{1}{k}t_k^2.  
	\end{equation*}
	Dividing both sides by $t_k$ and passing to the limit $k\to\infty$ we get 
	$$\langle\nabla\varphi_l(\bar x), v\rangle \le 0, \ \ l\in L.$$
	We claim that  there exists $l\in L$ such that  $\langle\nabla\varphi_l(\bar x), v\rangle=0$. If otherwise, then it follows from $(\bar\alpha, \bar\lambda, \bar\mu)\in\mathrm{A}(\bar x)$ 
	that
	\begin{align*}
		0= \langle\nabla_x\mathcal{L}(\bar x, \bar\alpha, \bar\lambda, \bar{\mu}), v\rangle&=\sum_{l\in L}\bar{\alpha}_l\langle \nabla \varphi_l(\bar x), v \rangle+\sum_{i\in I^0(\bar x)}\bar{\lambda}_i\langle\nabla g_i(\bar x), v\rangle
		\\
		&\ \ \ \ \ \ \ \ \ \ \ \  \  \ \ \ \ \ \ \ \  \ \ \ \ \ +\sum_{j\in J}\bar{\mu}_j\langle\nabla h_j(\bar x), v\rangle
		\\
		&=\sum_{l\in L}\bar{\alpha}_l\langle \nabla \varphi_l(\bar x), v \rangle+\sum_{i\in I^0(\bar x)}\bar{\lambda}_i\langle\nabla g_i(\bar x), v\rangle<0,
	\end{align*}
	a contradiction. Hence, $v\in [T(\bar x, X)\cap K(x)]\setminus\{0\}$.

	By  Theorem~\ref{Taylor_formula}, we can find $z_k\in \partial^2 \mathcal{L}(\xi_k, \bar\alpha,\bar \lambda, \bar \mu ) (t_kv_k)$, where $\xi_k\in [x_k,\bar x]$ such~that
	\begin{align*}
		\dfrac{1}{2} \langle z_k, x_k-\bar x\rangle \le \mathcal{L}(x_k, \bar\alpha, \bar \lambda, \bar \mu) - \mathcal{L}(\bar x, \bar\alpha, \bar \lambda, \bar \mu) - \langle \nabla_x\mathcal{L}(\bar x, \bar\alpha, \bar \lambda, \bar \mu) , x_k-\bar x \rangle.
	\end{align*}
	This and the fact that $x_k\in X$ imply that   
	\begin{equation*} 
		\aligned
		\dfrac{1}{2} \langle z_k, t_kv_k\rangle &\le \mathcal{L}(x_k, \bar\alpha, \bar \lambda, \bar \mu) - \mathcal{L}(\bar x, \bar\alpha, \bar \lambda, \bar \mu)- \langle \nabla_x\mathcal{L}(\bar x, \bar\alpha, \bar \lambda, \bar \mu) , x_k-\bar x \rangle
		\\
		&=\sum_{l\in L}\bar{\alpha}_l[\varphi_l(x_k)-\varphi_l(\bar x)]+\sum_{i\in I^0(\bar x)}\bar\lambda_i [g_i(x_k)-g_i(\bar x)]
		\\
		&\ \ \ \ \ \ \  \ \ \ \ \ \ \ \ \ \ \ \  \ \ \ \ \ \ \ \ \ \ \ \  \  +\sum_{j\in J}\bar{\mu}_j[h_j(x_k)-h_j(\bar x)]
		\\
		&\leq \frac{1}{k}\bigg(\sum_{l\in L}\bar{\alpha}_l\bigg)t_k^2.
		\endaligned
	\end{equation*}
	By Proposition~\ref{property}(i) and the fact that $z_k\in \partial^2 \mathcal{L}(\xi_k, \bar\alpha,\bar \lambda, \bar \mu ) (t_kv_k)$, there exists $z'_k~\in~\partial^2 \mathcal{L}(\xi_k, \bar\alpha,\bar \lambda, \bar \mu )  (v_k)$ such  that  $z_k=t_k z'_k$.
	Hence
	$$\dfrac{1}{2}t_k^2 \langle z'_k, v_k\rangle \leq   \frac{1}{k}\bigg(\sum_{l\in L}\bar{\alpha}_l\bigg)t_k^2,\ \ \forall k\in\mathbb{N},$$ 
	and this implies that $\langle z, v \rangle \le 0$ for some $z\in \partial^2 \mathcal{L}(x, \bar\alpha,\bar \lambda, \bar \mu ) (v) $, which contradicts the assumptions. 
\end{proof}

To illustrate Theorem \ref{suffi_MP}, we conclude this section with an example.
\begin{example}  \rm 
	Consider the problem~\eqref{MOP} with the following data:
	\begin{align*}
		\varphi_1(x)=x^2, \varphi_2(x)=\int_0^{x} |t|dt,\  g(x)=-x \ \ \forall x\in\mathbb{R},
		\\
		\varphi(x):=(\varphi_1(x), \varphi_2(x)), X:=\{x\in\mathbb{R}\,|\, g(x)\leq 0\},
	\end{align*}
	and $\bar x=0\in X$. First, we observe that $\bar x$ is a stationary point of~\eqref{MOP} with the Lagrange multipliers $ \bar \lambda =0, \bar \alpha=(1, 0)$. By simple calculation, we find
	$$T(\bar x, X)=\mathbb{R}_+$$
	and $K(\bar x)=\mathbb{R}_+$. So if $v\in [T(\bar x, X)\cap K(\bar x)]\setminus \{0\}$ then $v>0$.
	
	The Lagrangian associated with~\eqref{MOP} is
	$$\mathcal{L}(x,\alpha, \lambda)=\alpha_1 x^2  + \alpha_2 \int_0^{x} |t|dt -\lambda x.$$ 
	It is clear that  $\nabla_x \mathcal{L}(x, \alpha,\lambda)=2 \alpha_1 x +\alpha_2 |x|-\lambda$. So for any $v$  one has
	$$\langle \nabla_x \mathcal{L}(x,\bar \alpha, \bar\lambda), v \rangle= 2 v\bar \alpha_1 x +v\bar \alpha_2 |x|-\bar \lambda=2 v x .$$
	Since $\mathcal{L}$ is a Lipschitz function, from~\eqref{ct2.1}, we have
	\begin{align*}\partial^2 \mathcal{L}(x,\bar \alpha,\bar \lambda) (v) &=\partial \langle v, \nabla_x \mathcal{L}(.,\bar \alpha,\bar \lambda) \rangle (\bar x)=\{2v\}. 
	\end{align*}
	This implies that
	$$\langle z, v\rangle=2v^2>0$$
	for all $z\in \partial^2 \mathcal{L}(x, \bar \alpha,\bar \lambda)(v)$ and $v\in [T(\bar x, X)\cap K(\bar x)]\setminus \{0\}$. Thus, by Theorem~\ref{suffi_MP}, we conclude that $\bar x=0$ is an isolated local solution of order $2$ of \eqref{MOP}. 
\end{example}
%%%%%%%%%%%%%%%%
\section{Concluding remarks} 
This paper focuses on deriving second-order necessary and sufficient optimality conditions for $C^{1,1}$ optimization problems in finite-dimensional spaces, subject to both inequality and equality constraints. Our approach utilizes the Lagrangian function and the concept of limiting second-order subdifferentials to obtain these conditions. These results represent a significant extension of existing results for optimization problems with $C^{1,1}$  data.

\section*{Acknowledgments} This research is funded by Hanoi Pedagogical University 2 under grant number HPU2.2023-UT-11.

\section*{Disclosure statement} 
The authors declare that they have no conflict of interest.

\section*{Data availability} There is no data included in this paper.

\end{document}